\documentclass[12pt]{amsart}
\usepackage{amsmath, enumitem}
\usepackage{amsfonts}
\usepackage{amssymb}
\usepackage{amsthm}
\usepackage{graphicx}
\usepackage{soul}
\usepackage[usenames, dvipsnames]{color} 
\usepackage{hyperref}
\usepackage{fancyhdr}
\usepackage{bbm}
\usepackage[a4 paper, left=1in, right=1in, top=1in, bottom=1in]{geometry}
\usepackage[english]{babel}
\usepackage[protrusion=true,expansion=true]{microtype}	
\usepackage{amsmath,amsfonts,amsthm,amssymb,hyperref}
\usepackage{graphicx}

\newtheorem{Theorem}{Theorem}[section]

\newtheorem{Lemma}[Theorem]{Lemma}
\newtheorem{Proposition}[Theorem]{Proposition}
\newtheorem{proposition}[Theorem]{Proposition}
\newtheorem{theorem}{Theorem}[section]
\newtheorem{corollary}[Theorem]{Corollary}
\newtheorem{lemma}[Theorem]{Lemma}
\theoremstyle{definition}
\newtheorem*{definition*}{Definition}
\newtheorem{Definition}[Theorem]{Definition}

\newtheorem{Remark}[Theorem]{Remark}
\newtheorem{example}[Theorem]{Example}

\def \dist {\mathrm{dist}}

\def \R {\mathbb{R}}
\def \Rn {\mathbb{R}^n}

\def \T {\mathbb{T}}
\def \Z {\mathbb{Z}}

\def \e {\varepsilon}
\def \ep {\varepsilon}

\def \O {\Omega}
\def \G{\Gamma}

\def \grad {\nabla}

\def\Xint#1{\mathchoice
{\XXint\displaystyle\textstyle{#1}}%
{\XXint\textstyle\scriptstyle{#1}}%
{\XXint\scriptstyle\scriptscriptstyle{#1}}%
{\XXint\scriptscriptstyle\scriptscriptstyle{#1}}%
\!\int}
\def\XXint#1#2#3{{\setbox0=\hbox{$#1{#2#3}{\int}$ }
\vcenter{\hbox{$#2#3$ }}\kern-.6\wd0}}

\def\dashint{\Xint-}

\newcommand{\eref}[1]{(\ref{e.#1})}
\newcommand{\tref}[1]{Theorem \ref{t.#1}}
\newcommand{\lref}[1]{Lemma \ref{l.#1}}

\newcommand{\cref}[1]{Corollary \ref{c.#1}}
\newcommand{\fref}[1]{Figure \ref{f.#1}}
\newcommand{\sref}[1]{Section \ref{s.#1}}

\newcommand{\dref}[1]{Definition \ref{d.#1}}

\title[Quantitative convergence in oscillatory obstacle problem]{Quantitative convergence of the ``bulk'' free boundary in an oscillatory obstacle problem}
\author[F. Abedin and W. M. Feldman]{Farhan Abedin and William M. Feldman} 
\address{Department of Mathematics, Lafayette College, Easton, PA 18042}
\email{abedinf@lafayette.edu}
\address{Department of Mathematics, University of Utah, Salt Lake City, UT 84112}
\email{feldman@math.utah.edu}

\date{\today}

\begin{document}

\begin{abstract}
We consider an oscillatory obstacle problem where the coincidence set and free boundary are also highly oscillatory.  We establish a rate of convergence for a regularized notion of free boundary to the free boundary of a corresponding classical obstacle problem, assuming the latter is regular. The convergence rate is linear in the minimal length scale determined by the fine properties of a corrector function.
\end{abstract}

\maketitle

\section{Introduction}


Let $U \subset \Rn$ be a smooth, bounded domain, and let $\varphi_0 \in C^2(U) \cap C(\overline{U})$ be an obstacle that is positive somewhere in $U$, negative on $\partial U$, and satisfies the ellipticity condition
\begin{equation}\label{ellipticity}
\lambda \leq -\Delta \varphi_0 \leq \lambda^{-1}
\end{equation}
for some $1\geq\lambda > 0$. Consider the obstacle minimal supersolution above $\varphi_0$:
\begin{equation}\label{homogOP}
u_0(x) := \min\left\{ v :  \Delta v \leq 0 \text{ in } U, \ v \geq \varphi_0(x) \text{ in } U,  \ v \geq 0 \text{ on } \partial U \right\}.
\end{equation}
Then $u_0$ satisfies
\[ \min\{ \Delta u_0, u_0 - \varphi_0\} = 0.\]
The non-contact set of $u_0$ is $\O_0 := \{u_0 > \varphi_0\}\cap U$, the contact (or coincidence) set of $u_0$ is $\Lambda_0 := \{u_0 = \varphi_0\} \cap U$, and the free boundary is the set $\G_0:=\partial \{u_0 = \varphi_0\} \cap U$.  

In this work we study a natural toy model for the behavior of an elastic membrane resting on a rough surface. Let $\psi$ be $\Z^n$-periodic, $-1 \leq \psi \leq 0$, and let $p\in \R$ be a given exponent. For each $\e > 0$, define the rough obstacle
\[ \varphi_\ep(x) := \varphi_0(x) +\ep^p \psi(x/\ep).\] 
Consider the obstacle minimal supersolution
\begin{equation}\label{epsOP}
u_{\e}(x) = \min\left\{ v :  \Delta v \leq 0 \text{ in } U, \ v \geq \varphi_0(x) + \ep^p \psi(x/\ep) \text{ in } U,  \ v \geq 0 \text{ on } \partial U \right\}.
\end{equation}
which satisfies
\[ \min\{ \Delta u_\ep, u_\ep - \varphi_\ep\} = 0.\]
The non-contact set of $u_{\e}$ is $\O_{\e}:= \{u_{\e} > \varphi_{\e} \} \cap U$, the contact (or coincidence) set of $u_{\e}$ is $\Lambda_{\e} := \{u_{\e} = \varphi_{\e}\} \cap U$, and the free boundary is the set $\G_{\e}:=\partial \{u_{\e} = \varphi_{\e}\} \cap U$.

\begin{figure}[h]
  \includegraphics[width=.4\textwidth,height = .35\textwidth]{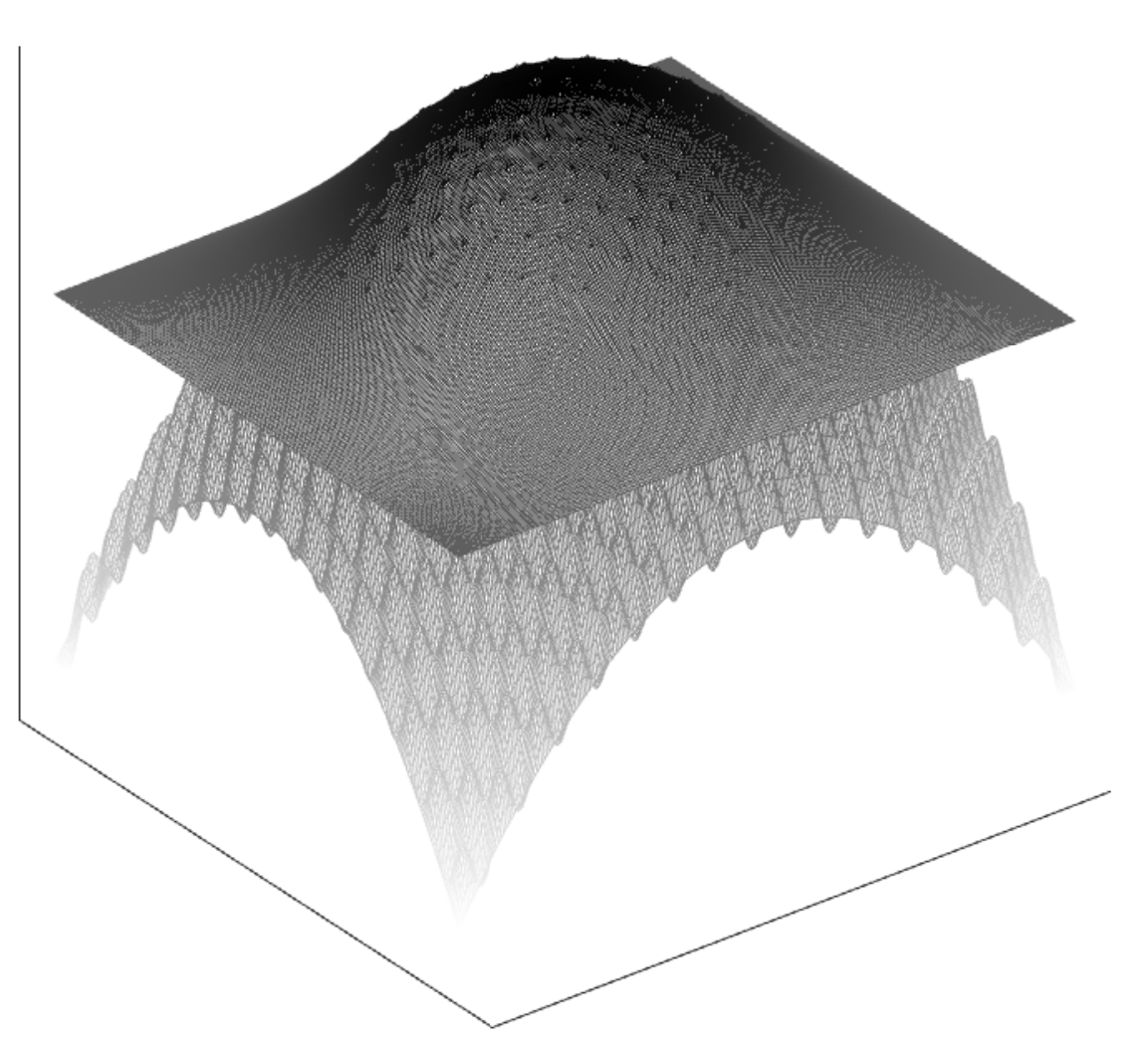}
  \hspace{.5in}
  \includegraphics[width=.4\textwidth,height = .35\textwidth]{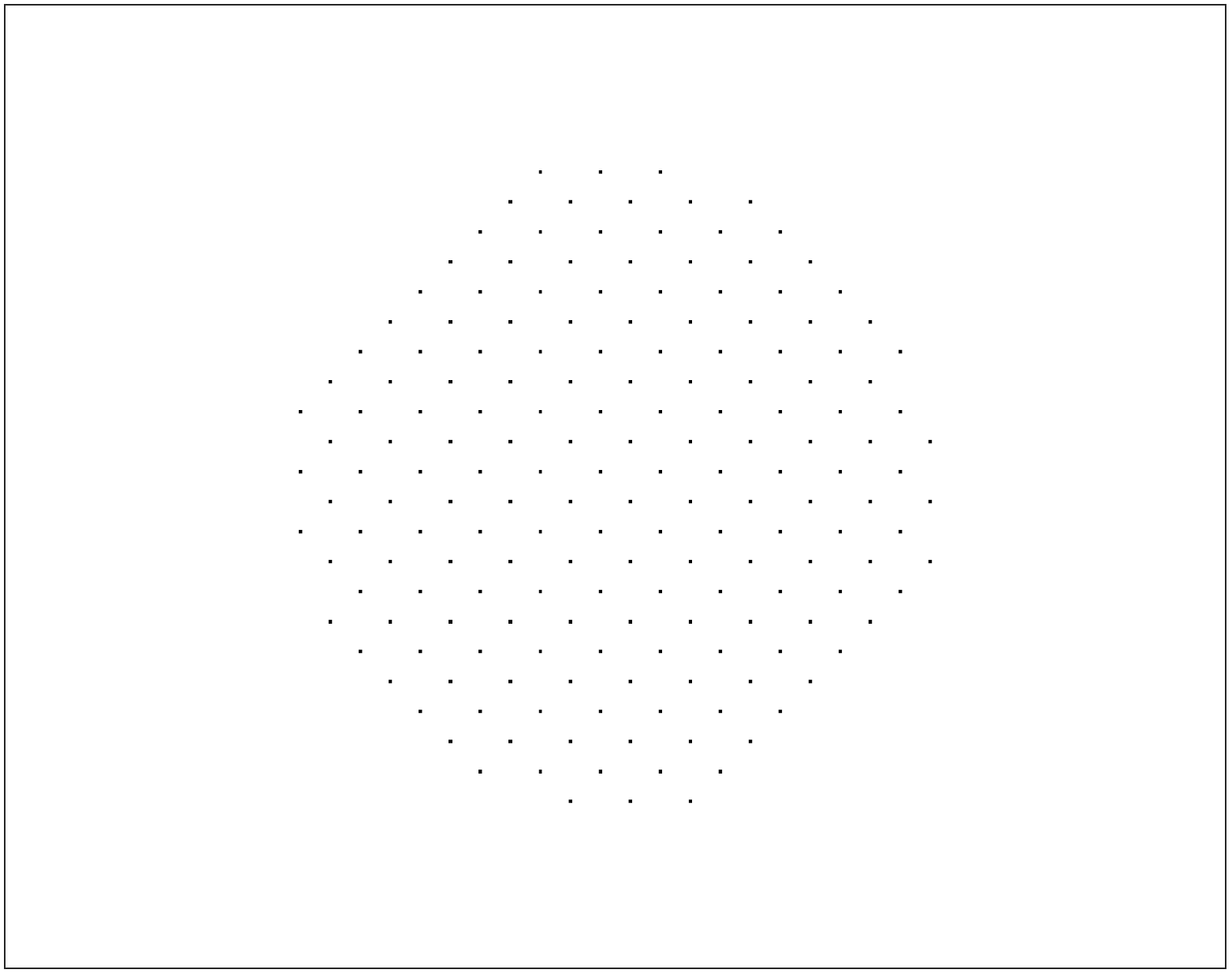}
  \caption{Left: simulation of obstacle problem solution above a parabolic obstacle perturbed by an oscillating sinusoid.  Right: contact set of the solution with the obstacle.}
  \label{f.obsfig}
\end{figure}

Our goal is to quantitatively compare the functions $u_{\e}$ and $u_0$, as well the contact sets $\Lambda_\ep$ and $\Lambda_0 $. 
Note that the obstacle $\varphi_{\e}$ and contact set $\Lambda_\ep$ may be highly oscillatory. Generally speaking, when $p < 2$ the obstacle solution $u_{\e}$ rests on the peaks of $\psi$ and the contact set is effectively ``discretized", see \fref{obsfig}. Furthermore, as illustrated by the example in \fref{obsfig}, one cannot expect the free boundary $\G_{\e}$ of the oscillatory obstacle problem to converge in Hausdorff distance to the free boundary $\G_0$ of the unperturbed obstacle problem.  

The purpose of this note is to show that certain analogues of the basic regularity theory for the classical obstacle problem can be developed for the oscillatory obstacle problem and used to define a notion of \emph{bulk contact set} $\tilde{\Lambda}_\ep$ and \emph{bulk free boundary} $\tilde{\G}_\ep$, which can be compared directly with the the effective contact set $\Lambda_0$ and $\G_0$.  The rate of convergence of $\tilde{\G}_\ep$ to $\G_0$ is determined by fine properties of a corrector-type function arising from an appropriate cell problem, which is studied in Section \ref{sec:corrector}.

Let us state our main result more precisely.  We first introduce the corrector function. For each $\mu > 0$, let $\chi_{\mu}$ be the $\Z^n$-periodic minimal supersolution of the problem
\begin{equation}\label{correctorproblem}
\chi_\mu := \min\left\{ v :  \Delta v \leq \mu \text{ in } \R^n, \ v \geq \psi(x) \text{ in } \R^n,  \ v \ \hbox{ is $\Z^n$-periodic} \right\}
\end{equation}
Since the zero function is a supersolution for each $\mu > 0$, we have $\psi \leq \chi_{\mu} \leq 0$ in $\Rn$. Define \begin{equation}\label{heightofchimu}
\mathcal{E}(\mu) = -\inf \chi_\mu = \|\chi_\mu\|_\infty.
\end{equation}
We show in Section \ref{sec:corrector} that $\mathcal{E}(\mu) \to 0$ as $\mu \to 0$ but the rate depends sensitively on the behavior of $\psi$ near its maxima.  In some cases this rate is linear in $\mu$, but it also may be H\"older or worse.

Next we define the minimal length scale coming from the corrector function; this is the quantity that will determine the rate of convergence in our main result.
\begin{Definition}
The \emph{minimal length scale} of the $\e$-oscillatory obstacle problem is
\begin{equation}\label{minlengthscale}
\mathfrak{r}(\ep) := \left(\e^{p}\mathcal{E}(\lambda^{-1}\ep^{2-p}) \right)^{1/2}.
\end{equation}
\end{Definition}
We will assume henceforth that $\lim_{\e \to 0} \mathfrak{r}(\ep) = 0$; this is a requirement on the exponent $p$.  The condition $p \geq 0$ is always sufficient, but we may consider $p <0$ as well. For instance, when $\mathcal{E}(\mu)= O(\mu)$, this holds for all $p \in\R$, while when $\mathcal{E}(\mu) = O(\mu^\alpha)$, this holds when $(1-\alpha)p + 2 \alpha >0$.

Finally, we define the notion of ``bulk'' free boundary for the $\e$-oscillatory problem, which we consider to be an appropriate proxy for the free boundary $\G_{\e}$ when establishing quantitative convergence results.

\begin{Definition}\label{d.bulk}
The \emph{bulk contact set} of the $\e$ obstacle problem, denoted $\tilde{\Lambda}_{\e}$, is the union of cubes in the  $4(\lambda^{-1}2n)^{\frac{1}{2}}\mathfrak{r}(\e) \Z^n$ lattice that intersect $\Lambda_\ep$. The \emph{bulk free boundary} of the $\e$ obstacle problem is the set $\tilde{\G}_{\e} := \partial \tilde{\Lambda}_{\e} \cap U$.
\end{Definition}

We can now state our main result.

\begin{theorem}\label{thm:main}
Assume $\G_0$ consists only of regular points, in the sense of Caffarelli \cite{CaffarelliObstacleRevisited}. Then there exists $C \geq 1$ depending on the solution of \eqref{homogOP} such that for all $\e \leq 1$,
 \[d_H(\Lambda_0, \tilde{\Lambda}_{\e}) \leq C \mathfrak{r}(\e) \qquad \text{and} \qquad d_H(\Gamma_0,\tilde{\Gamma}_\ep) \leq C\mathfrak{r}(\ep).\]
 \end{theorem}
 

We refer to Section 4 for the precise regularity properties of $\Gamma_0$ that are assumed.  We also prove a rate of convergence of the gradients which can be found in \sref{gradient}.

\begin{Remark} The estimate $d_H(\Lambda_0, \Lambda_{\e}) \leq C \mathfrak{r}(\e)$ also holds, the notion of bulk coincidence set / bulk free boundary are really needed for comparing the free boundaries. 
\end{Remark}

\subsection{Literature}  The obstacle problem is a classical and much studied example of PDE problem featuring a free boundary.  It was realized some time ago, maybe first by De Giorgi, Dal Maso and Longo \cite{DeGiorgi}, that the $\Gamma$-limit of an obstacle-type minimization problem may be of a different type depending critically on the capacity of the peaks of the obstacle.  This phenomenon has been studied significantly in \cite{CarboneColombini,AttouchPicard1,AttouchPicard2,DalMasoLongo,CaffarelliLee,CioranescuMurat}.  

There is also  work on the stability of the obstacle problem under perturbations of the obstacle \cite{Blank, BlankLeCrone, CaffarelliBUMI, SerfatySerra}. The primary difference between these works and ours is that the stability is measured with respect to strong norms on the Laplacian of the obstacle ($L^\infty$ or $L^1$). In our case the perturbation of the Laplacian of the obstacle is only small in a weak / negative order sense.  



One motivation for studying the oscillatory obstacle problem \eqref{epsOP} is its connection with the Hele-Shaw flow in periodic media \cite{ElliotJanovsky,Kim, KimMellet}. The obstacle problem studied in \cite{KimMellet} resembles \eqref{epsOP} when $p = 2$, but the nature of the transformation from the Hele-Shaw problem precludes the kind of oscillatory contact set that we are interested in here.


The works that are closest to ours are \cite{CodegoneRodrigues} and \cite{AleksanyanKuusi}. Our results quantify the qualitative convergence results obtained in \cite{CodegoneRodrigues}, and we go further by establishing the convergence of the bulk free boundary of the oscillatory obstacle problem to the free boundary of the unperturbed problem. The recent work \cite{AleksanyanKuusi} establishes a large scale regularity theory for the obstacle problem with an oscillatory divergence-form elliptic operator.  They make an assumption on ``compatibility" of the obstacle with the operator, which avoids the kind of oscillatory contact set that we study here.  However, removing this compatibility assumption in the context of oscillatory divergence form PDE operator would result in a cell problem that is more singular and apparently much more difficult than the one we study in Section \ref{sec:corrector}, so it is not clear if the notion of bulk contact set would be useful in that context.



\subsection{Acknowledgments} F.A. acknowledges support from the AMS and Simons Foundation through an AMS-Simons Travel Grant. A significant portion of this work was carried out while F.A. held a postdoctoral position in the Department of Mathematics at the University of Utah, whose support is also gratefully acknowledged. W.F. acknowledges the support of NSF grant DMS-2009286.

\section{Corrector and Cell Problem}\label{sec:corrector}

This section is devoted to the study of a cell problem that is meant to describe the local behavior of the $\e$-obstacle solution $u_{\e}$ above its contact set $\Lambda_{\e}$.  Recall the corrector function $\chi_{\mu}$ defined in \eqref{correctorproblem} and its height $\mathcal{E}(\mu)$ defined in \eqref{heightofchimu}.
Our goal is to show (with quantitative estimate) that $\mathcal{E}(\mu) \to 0$ as $\mu \to 0$. 

We point out that a simple integration-by-parts argument yields the following bound on the Dirichlet energy in a unit cell:
\[ \int_{\T^n} |\nabla \chi_\mu|^2 dx = \int_{\T^n} (-\chi_\mu) \Delta \chi_\mu dx \leq \mu \mathcal{E}(\mu).\]
The rate of convergence $\mathcal{E}(\mu) \to 0$ is sensitive to the structure of $\psi$ near its zero level set, in particular, the co-dimension of the zero level set and the regularity of $\psi$ near the zero level set. 

\begin{Theorem}\label{t.emubound}
The corrector $\chi_\mu \to 0$ as $\mu \to 0$ with the following estimates
\begin{enumerate}
\item In $n=2$ and for any $\alpha \in (0,1)$ there is a constant $C$ depending on $\|\psi\|_{C^{0,\alpha}}$ so that
\[ \mathcal{E}(\mu) \leq C\mu(1+|\log \mu|).\]

\item In $n \geq 3$ there is a constant $C$ depending on $\|\psi\|_{C^{1,1}}$ so that
\[ \mathcal{E}(\mu) \leq C\mu^{2/n}\]
or more generally depending on $\|\psi\|_{C^s} = \|\psi\|_{C^{[s],s-[s]}}$ 
\[ \mathcal{E}(\mu) \leq C(\|\psi\|_{C^s})\mu^{\frac{s}{s+n-2}} \ \hbox{ for } \ s \in (0,2].\]
\item If $\partial \{\psi<0\}$ contains a regular submanifold $\Sigma$ of codimension $k \in \{1,\dots,n\}$ then
\[\mathcal{E}(\mu) \leq \begin{cases} C\mu & k =1\\
C\mu(1+|\log \mu|) & k =2\\
C\mu^{\frac{2}{k}} & 2 < k \leq n
\end{cases} \]
where the constants $C$ depend on $\|\psi\|_{C^{1,1}}$ and the regularity of the parametrization of $\Sigma$.
\end{enumerate}
\end{Theorem}
\begin{example}
If $\psi : \R^n \to \R$ is laminar, i.e. it only depends on $m<n$ variables, then part (3) applies with $k = m$.
\end{example}

\begin{proof}
Let $G(x)$ be the Green's function for the Laplace operator on the torus $\mathbb{T}^n = \R^n \mod \Z^n$ solving
\[ \Delta G = 1 - \sum_{k \in \Z^n}\delta_k\]
and normalized so that $\min G = 0$. Standard Green's function estimates give
  \[ A = \sup_{B_{1/2}(0)} |G(x) + \frac{1}{2\pi}\log |x| |< + \infty \ \hbox{ in } \ n = 2\]
  and
\[ A = \sup_{B_{1/2}(0)} |G(x) - \alpha_n|x|^{2-n}|< + \infty \ \hbox{ in } \ n \geq 3.\]
Suppose, without loss, that $0 \in \{\psi = 0\}$ and
\[ \psi(x) \geq - B|x|^2 \ \hbox{ for } \ |x| \leq 1/2,\]
this is always true for some $B > 0$ because $\psi$ is smooth and achieves its maximum at $0$.  

Then define, for some $r < 1/2$ to be chosen,
\[ h(x) = \mu(G(x) +\frac{1}{2\pi}\log(r) - A) - Br^2 \ \hbox{ in } \ n =2\]
or
\[ h(x) = \mu(G(x) - \alpha_nr^{2-n} - A) - Br^2 \ \hbox{ in } \ n \geq 3\]
so that 
\[ h(x) \leq \psi(x) \ \hbox{ on } \ \partial B_r(0).\]
By the comparison principle,
\[ h(x) \leq \chi_\mu(x) \ \hbox{ in } \ \T^n \setminus B_r(0).\]
Note that, in $n \geq 3$,
\[ -\min h \leq \alpha_n\mu r^{2-n} + A\mu +Br^2\]
the right hand side is minimized when $\mu = r^n$ and
\[ - \min h \lesssim \mu^{\frac{2}{n}}.\]
When $n=2$
\[ -\min h \leq \mu \frac{1}{2\pi}|\log r| + A\mu +Br^2\]
so we choose $r^2 = \mu$ to get
\[ -\min h \lesssim \mu (1+|\log \mu|).\]
Note that in the $n=2$ case the estimate would be of the same type even if we had only assumed $\psi(x) \geq - Br^\alpha$ for some other value of $\alpha \in (0,\infty)$.

The lower bound for $\chi_\mu$ obtained above only holds in $\T^n \setminus B_r(0)$, but in $B_r(0)$ we still have $\chi_\mu \geq \psi \geq - Br^2$ which is a lower bound of the same order.

Next we consider the case when $\partial \{\psi < 0\}$ contains a regular submanifold $\Sigma$ of co-dimension $k$.   Then define
\[ g(x) = \int_{\Sigma} G(x-y) dS(y).\]
Then $\Delta g = 1$ in $\T^n \setminus \Sigma$ and standard integral estimates show that
\[ g(x) \leq \begin{cases}C &k =1 \\
C(1+|\log d(x,\Sigma)|) & k=2 \\
Cd(x,\Sigma)^{2-k} & 3 \leq k \leq n.
\end{cases}\]
Then we do the same barrier argument as in the previous argument replacing $G$ with $g$.
\end{proof}

\section{$L^{\infty}$ Estimates for the Obstacle Solutions}

Our goal in this section is to obtain $L^{\infty}$ estimates for the difference of $u_0$ and $u_{\e}$; an important result we will obtain along the way is the non-degeneracy property \lref{nondegen}. It will be convenient, at this stage, to work with appropriate height functions for each of the obstacle problems \eqref{homogOP} and \eqref{epsOP}. 
%

Let $w_0:= u_0 - \varphi_0$ be the height function for the obstacle problem \eqref{homogOP}. Then $w_0$ solves the obstacle problem
\begin{equation}\label{e.homogOPheight}
w_0(x) = \min\left\{ v :  \Delta v \leq -\Delta \varphi_0 \text{ in } U, \ v \geq 0 \text{ in } U,  \ v \geq -\varphi_0 \text{ on } \partial U \right\}.
\end{equation}
We note that $\Omega_0= \{w_0>0\} \cap U$ and $\Lambda_0 = \{w_0 =0 \} \cap U$.  From the theory for the classical obstacle problem \cite{CaffarelliObstacleRevisited}, we know that $w_0$ is $C^{1,1}$ and satisfies
\[ \Delta w_0 = -\Delta \varphi_0 \ \hbox{ in } \ \Omega_0, \ \hbox{ and } \ w_0 = |Dw_0| = 0 \ \hbox{ on } \ \Lambda_0.\]
Next consider the function $w_{\e}:= u_{\e} - \varphi_0$ which solves the obstacle problem
\begin{equation}\label{e.epOPheight}
w_\ep(x) = \min\left\{ v :  \Delta v \leq -\Delta \varphi_0 \text{ in } U, \ v \geq \ep^p\psi(x/\ep) \text{ in } U,  \ v \geq -\varphi_0 \text{ on } \partial U \right\}.
\end{equation} 
 Although $w_{\e}$ can be negative, we refer to it as a ``height function" for the oscillatory problem.  Note that 
 \[ \Omega_\ep = \{w_\ep > \ep^p\psi(x/\ep)\} \cap U \ \hbox{ and } \ \Lambda_\ep = \{w_\ep = \ep^p\psi(x/\ep)\} \cap U.\]
Also, $w_\ep$ satisfies
 \[ \Delta w_\ep = - \Delta \varphi_0 \ \hbox{ in } \ \Omega_\ep.\]
Certainly, $w_\ep \geq \ep^p \psi(x/\ep)$, but there is actually a much stronger lower bound in terms of the corrector as presented in the next Lemma.

\begin{lemma}\label{lem:correctedobstacle}
Let $w_\ep = u_\ep - \varphi_0$ as above then
$$w_{\e}(x) \geq\e^p\chi_{\lambda^{-1}\e^{2-p}}(x/\e)\quad \text{ for all } x \in U.$$
\end{lemma}

%

\begin{proof}

  Consider the function $v_\ep(x) =\ep^p \chi_{\lambda^{-1}\ep^{2-p}}(x/\ep) - w_\ep(x)$ and the set $V = \{v_\ep >0\}$.  Since $w_\ep >0$ on $\partial U$ we have $V \subset\subset U$.  Thus $v_\ep$ vanishes on $\partial V \cap U$.  

For any $x \in V$ we have $\ep^p \chi_{\lambda^{-1}\ep^{2-p}}(x/\ep) > w_\ep(x) \geq \ep^p \psi(x/\ep)$.  So by \eqref{correctorproblem} \[\left(\Delta \e^{p} \chi_{\lambda^{-1}\e^{2-p}}\right) \left(\tfrac{x}{\e} \right) = \lambda^{-1}.\]  Since $w_\ep$ satisfies $\Delta w_\ep \leq -\Delta \varphi_0(x) \leq \lambda^{-1}$ in $U$, it follows that $v_\ep$ is subharmonic in $V$ and vanishes on $\partial V$, from which it follows that $v_\ep \equiv 0$ in $V$, implying $V$ is empty.
\end{proof}

 As a consequence, we have the following estimate for the difference of the height functions $w_0$ and $w_{\e}$. Note that, in many cases, this is a significant improvement on the trivial $L^\infty$ estimate between $w_\ep$ and $w_0$ of order $\ep^p$; this is because, recalling the definition of $\mathcal{E}(\mu)$ from \eqref{heightofchimu}, $\min_U\e^p\chi_{\lambda^{-1}\e^{2-p}}(x/\e) = -\mathfrak{r}(\e)^2$ as long as $U$ contains a single $\ep\Z^n$ periodic cell.

\begin{proposition}\label{p.Linftyrate}
Let $\mathfrak{r}(\ep)$ as defined in \eqref{minlengthscale} then
\[ w_0 - \mathfrak{r}(\e)^2 \leq w_\ep \leq w_0 \quad \text{ in } U.\]
\end{proposition}

\begin{proof}  To prove $w_\ep \leq w_0$, we observe that $w_0 \geq 0 \geq \ep^p\psi(x/\ep)$ in $U$ and $w_0 = \varphi_0$ on $\partial U$. Therefore $w_0$ is admissible for the minimization \eref{epOPheight} and so $w_0 \geq w_\ep$. 

Next, to show $w_0 - \mathfrak{r}(\e)^2 \leq w_\ep$, we observe that, by translation invariance, the solution of the obstacle problem on $U$ with boundary condition $-\varphi_0-\mathfrak{r}(\e)^2$ and obstacle $ - \mathfrak{r}(\e)^2$ is $w_{0} - \mathfrak{r}(\e)^2$. Since $w_{\e} \geq \ep^{p}\chi_{\lambda^{-1}\ep^{2-p}}(x/\ep) \geq  - \mathfrak{r}(\e)^2$ in $U$ and $w_{\e} = -\varphi_0 \geq \varphi_0-\mathfrak{r}(\e)^2$ on $\partial U$, we conclude that $w_0 - \mathfrak{r}(\e)^2 \leq w_{\e}$ in $U$. 
\end{proof}

\subsection{Non-degeneracy} It is well-known that the height function $w_0$ satisfies the following non-degeneracy property: for all $z \in \G_0$ and $r > 0$ such that $B_r(z) \Subset U$, we have
\begin{equation}\label{e.0nondegen}
\sup_{B_r(z)} w_0  \geq \tfrac{\lambda}{2n} r^2 .
\end{equation}
For the height function $w_{\e}$ of the oscillatory problem, we will establish an analogous non-degeneracy statement at scales larger than $\mathfrak{r}(\ep)$. 

\begin{Lemma}\label{l.nondegen} For all $z \in U$ with $\dist(z, \Lambda_\ep) >  (\lambda^{-1}2n)^{\frac{1}{2}} \mathfrak{r}(\e)$ and $r > 0$ such that $B_r(z) \Subset U$, we have
$$\sup_{B_r(z)} w_{\e} \geq \tfrac{\lambda}{2n} r^2 - \mathfrak{r}(\ep) ^2.$$
\end{Lemma}
\begin{proof}
On the set $D := B_r(z) \cap \Omega_\ep$, consider the function
$$\zeta_{\e}(x):= w_{\e}(x) - \tfrac{\lambda}{2n}|x-z|^2.$$
Then, by \eqref{ellipticity},
$$\Delta \zeta_{\e} = \Delta w_{\e} - \lambda = - \Delta \varphi_0 - \lambda \geq 0 \ \hbox{ in } \ D.$$
The maximum principle implies $\zeta_{\e}$ attains its maximum on $\partial D$. Furthermore, since $\zeta_{\e}(z) = w_{\e}(z) \geq - \mathfrak{r}(\e)^2$, we have $\max_D \zeta_{\e} \geq - \mathfrak{r}(\e)^2$. 

Now let $x_{\text{max}} \in \partial D$ be such that $\zeta_{\e}(x_{\text{max}}) = \max_D \zeta_{\e}$. We decompose $\partial D$ as the disjoint union $\partial D = (\partial B_r(z) \cap \O_{\e}) \cup (\overline{B_r(z)} \cap \Lambda_{\e})$. If $x_{\text{max}} \in \partial B_r(z) \cap \O_{\e}$, we have $|x_{\text{max}} - z| = r$, and so
$$-\mathfrak{r}(\e)^2 \leq \zeta_{\e}(x_{\text{max}}) = w_{\e}(x_{\text{max}}) - \tfrac{\lambda}{2n}r^2.$$
Consequently, $w_{\e}(x_{\text{max}}) \geq \frac{\lambda}{2n}r^2 - \mathfrak{r}(\e)^2$, from which it follows that $\sup_{B_r(z)} w_{\e} \geq \frac{\lambda}{2n}r^2 - \mathfrak{r}(\e)^2$ as claimed. 

If, on the other hand, we have $x_{\text{max}} \in \overline{B_r(z)} \cap \Lambda_\ep$ then
$$-\mathfrak{r}(\e)^2 \leq \zeta_{\e}(x_{\text{max}}) = w_{\e}(x_{\text{max}}) - \tfrac{\lambda}{2n}|x_{\text{max}} - z|^2  \leq - \tfrac{\lambda}{2n}|x_{\text{max}} - z|^2.$$
This implies $\frac{\lambda}{2n} |x_{\text{max}} - z|^2\leq \mathfrak{r}(\e)^2$, which contradicts the assumption $\dist(z,\Lambda_\ep)  >  (\lambda^{-1}2n)^{\frac{1}{2}} \mathfrak{r}(\e)$.
\end{proof}

\begin{Remark}\label{remark:rate}
The scale $\mathfrak{r}(\ep)$ is highly dependent on $\psi$, see \tref{emubound} above. Of particular interest is when $\mathfrak{r}(\ep) = O(\ep)$, which holds when $\mathcal{E}(\mu) = O(\mu)$ and for any value of $p \in \R$.  We believe that $\mathfrak{r}(\ep)$ is the ``correct" length scale to measure the contact set.  This is essentially because we postulate 
that the dominant term in the asymptotic expansion of the height function $w_\ep$ above the ``bulk" contact set is  the corrector $\e^p\chi_{\e^{2-p}}(x/\e)$, which has height scaling $\mathfrak{r}(\ep)^2$.  In order to grow away from this corrector via quadratic non-degeneracy at this same height scaling one needs to move distance $\mathfrak{r}(\ep)$ away from the ``bulk" contact set.
\end{Remark}

\section{Distance Estimates for the Free Boundaries}

In this final section, combine the results from the previous sections to prove Theorem \ref{thm:main}. Before we can do this, it will be necessary to make precise the regularity assumptions we make on the free boundary $\G_0$.

A well-known consequence of the classical regularity theory for the obstacle problem \cite{CaffarelliObstacleRevisited} is  a $C^{1,1}$ estimate for the height function $w_0$:

\begin{enumerate}[label = (\roman*)]
\item $C^{1,1}$ bound: $\sup_U |D^2w_0| \leq M$ with $M$ depending on the lower bound of $-\varphi_0$ on $\partial U$, on $\lambda$, and on $\|\Delta \varphi_0\|_{C^\gamma}$.
\end{enumerate}
We will also assume that $\G_0$ consists only of regular points in the sense of Caffarelli \cite{CaffarelliObstacleRevisited}. This leads to the following regularity properties:
\begin{enumerate}[label = (\roman*),resume]
\item Strong Non-Degeneracy: there exists $c_1 > 0$ such that if $x \in \O_0$ then $w_0(x) \geq c_1 d(x,\G_0)^2$. 

\item Uniform Positive Density of Contact Region: there exists a constant  $c_2 \in (0,\frac{1}{2})$ such that for any $r>0$ and $x \in \Lambda_0$, there exists $y \in \Lambda_0 \cap B_r(x)$ such that $B_{c_2 r}(y) \Subset  \Lambda_0 \cap B_r(x)$. 

\end{enumerate}

Both properties follow from well-known regularity results for the classical obstacle problem. We make some convenient citations: for (ii) apply \cite[Lemma 5.5]{AleksanyanKuusi}, for (iii) apply \cite[Theorem 7]{CaffarelliObstacleRevisited} and a compactness argument.

We recall the definition of bulk free boundary from \dref{bulk}.
\begin{definition*}\label{bulkfreeboundary}
The \emph{bulk contact set} of the $\e$ obstacle problem, denoted $\tilde{\Lambda}_{\e}$, is the union of cubes in the  $4(\lambda^{-1}2n)^{\frac{1}{2}}\mathfrak{r}(\e) \Z^n$ lattice that intersect $\Lambda_\ep$. The \emph{bulk free boundary} of the $\e$ obstacle problem is the set $\tilde{\G}_{\e} := \partial \tilde{\Lambda}_{\e} \cap U$.
\end{definition*}

Note that if $x \in \tilde{\G}_{\e}$, then $x$ belongs to some $4(\lambda^{-1}2n)^{\frac{1}{2}}\mathfrak{r}(\e) \Z^n$ lattice cube $Q_x$ and there is a neighboring $4(\lambda^{-1}2n)^{\frac{1}{2}}\mathfrak{r}(\e) \Z^n$ lattice cube $\hat{Q}_{x}$ such that $x \in \partial Q_x \cap \partial \hat{Q}_{x}$ and $\hat{Q}_{x} \subset \Omega_\ep$. The center $z$ of $\hat{Q}_{x}$ then satisfies $\text{dist}(z,\Lambda_\ep) \geq 2(\lambda^{-1}2n)^{\frac{1}{2}}\mathfrak{r}(\e) > (\lambda^{-1}2n)^{\frac{1}{2}}\mathfrak{r}(\e) $. In particular \lref{nondegen} can be applied at $z$.

The following non-degeneracy statement at the bulk free boundary is an immediate consequence of \lref{nondegen}. Such a non-degeneracy property can be viewed as an essential attribute of a ``good" notion of bulk free boundary.

\begin{corollary}\label{c.nondegen2}
There is $c(n,\lambda)>0$ so that for all $r > 0$, if $x \in \tilde{\Gamma}_\ep$ then 
\[ \sup_{B_r(x)} w_\ep \geq c(n,\lambda) r^2 - 2\mathfrak{r}(\ep)^2.\]
\end{corollary}
\begin{proof}
Let $z$ be as defined in the preceding paragraph. For $r \leq 2|x-z|$ use $w_\ep \geq - \mathfrak{r}(\ep)^2$, for $r \geq 2|x-z|$ use \lref{nondegen} centered at $z$.
\end{proof}

The properties (i), (ii) and (iii) of $\G_0$ are sufficient to derive an $\mathfrak{r}(\ep)$ rate of convergence of $\tilde{\Gamma}_\ep$ to $\Gamma_0$ in the Hausdorff distance, thus proving Theorem \ref{thm:main}.

\begin{Proposition}
There exists $C  > 0$ depending on $n,\lambda$ and the quantity $c_1$ from property (ii) such that for all $\e \leq 1$ and $x \in \tilde{\G}_{\e}$, we have
$$d(x,\G_0) \leq C\mathfrak{r}(\e).$$
\end{Proposition}

\begin{proof} Let $r= d(x,\Gamma_0)$. There are two possibilities:

\emph{Case 1:} $B_r(x) \subset \O_0$.


Let $Q_x$ be the $4(\lambda^{-1}2n)^{\frac{1}{2}}\mathfrak{r}(\e) \Z^n$ lattice cube containing $x$. By definition of $\tilde{\G}_{\e}$ there is a $4(\lambda^{-1}2n)^{\frac{1}{2}}\mathfrak{r}(\e) \Z^n$ lattice cube $Q_x$ such that $x \in \partial Q_x$ and $Q_x \cap \Lambda_\ep \neq \emptyset$; i.e. there is $z \in Q_x$ such that $w_{\e}(z) = \ep^p \psi(z/\ep) \leq 0$. 

Applying the strong non-degeneracy of $w_0$ at $z$,  we have
$$c_1 d(z,\G_0)^2 \leq w_0(z) \leq \mathfrak{r}(\e)^2 +w_{\e}(z) \leq \mathfrak{r}(\e)^2.$$
Thus $ d(z,\G_0) \leq c_1^{-1/2} \mathfrak{r}(\e)$. Since $d(x,z) \leq c(n,\lambda) \mathfrak{r}(\e)$, it follows that $d(x,\G_0) \leq C\mathfrak{r}(\e)$ as claimed.

\emph{Case 2:} $B_r(x) \subset \Lambda_0$.


Applying \cref{nondegen2}
$$cr^2 - 2\mathfrak{r}(\e)^2 \leq \sup_{B_{r}(x)} w_{\e} = \sup_{B_{r}(x)} (w_{\e} - w_0) \leq \mathfrak{r}(\e)^2.$$
It follows that $d(x,\G_0) \leq C \mathfrak{r}(\e)$.

\end{proof}

\begin{Proposition}
There is $C\geq1$ depending on $n$, $\lambda$ and the parameters from property (i) and (iii) above so that for all $x \in \G_0$
$$d(x,\tilde{\G}_{\e}) \leq C\mathfrak{r}(\e).$$
\end{Proposition}

\begin{proof} Fix $\e \leq \e_0$ and let $r= d(x,\tilde{\G}_{\e})$. There are two possibilities:

\emph{Case 1:} $B_r(x) \subset U \setminus \tilde{\Lambda}_{\e}$.

By the uniform positive density of $\Lambda_0$, we know there exists $y \in  \Lambda_0 \cap B_r(x)$ such that $B_{c_2 r}(y) \Subset  \Lambda_0 \cap B_r(x)$. We may also assume that $\text{dist}(y,\Lambda_\ep) > (\lambda^{-1}2n)^{\frac{1}{2}}\mathfrak{r}(\e)$, for otherwise we would already have $r \leq C \mathfrak{r}(\e)$. Applying Lemma \ref{l.nondegen} in $B_{c_2 r}(y)$, we find
$$c(c_2r)^2  -2 \mathfrak{r}(\e)^2 \leq \sup_{B_{c_2 r}(y)} w_{\e} = \sup_{B_{c_2 r}(y)} (w_{\e} - w_0) \leq \mathfrak{r}(\e)^2.$$
It follows that $r \leq C \mathfrak{r}(\e)$.

\emph{Case 2:} $B_r(x) \subset \tilde{\Lambda}_{\e}$.

Suppose $w_0$ attains its maximum on $B_r(x)$ at $x_{\text{max}} \in B_r(x)$. Assume without loss of generality that $r \geq \mathfrak{r}(\e)$. By definition of $\tilde{\Lambda}_{\e}$, we can find a point $y$ such that $|x_{\text{max}}-y| \leq C\mathfrak{r}(\e)$ and $w_{\e}(y) \leq 0$. Therefore, by the non-degeneracy of $w_0$ and the Lipschitz estimate for $w_0$ (from property (i) above), we have
\begin{align*}
Cr^2 \leq \sup_{B_r(x)} w_{0} & = w_0(x_{\text{max}}) \\
& = (w_0(x_{\text{max}}) - w_0(y)) + (w_0(y)  - w_{\e}(y)) +w_{\e}(y) \\
& \leq 2M r \mathfrak{r}(\e) + \mathfrak{r}(\e)^2 \leq (2M+1) r \mathfrak{r}(\e).
\end{align*}
Consequently, $r \leq C \mathfrak{r}(\e)$ as claimed.
\end{proof}

We remark that the Hausdorff distance estimate
\[ d_H(\Lambda_0,\tilde{\Lambda}_\ep) \leq C \mathfrak{r}(\ep)\]
holds by the same arguments as Case 1 of the previous two propositions.

\section{Gradient convergence}\label{s.gradient}

In this section we show another notion of convergence at the level of the gradient. Specifically we show that
\[  \left(\dashint_{B_{\mathfrak{r}(\ep)}(x)}|\grad w_0 - \grad w_\ep|^2 dy\right)^{1/2} \leq C \mathfrak{r}(\ep) \ \hbox{ in } \ \{ x \in U: d(x,\partial U) \geq \mathfrak{r}(\ep)\}. \] 
This can be considered as an $L^\infty$ estimate at scales above $\mathfrak{r}(\ep)$.

First of all note that $w_\ep - w_0$ is harmonic in the complement of $\Lambda_\ep \cup \Lambda_0$ so
\[ |\grad w_\ep(x) - \grad w_0(x)| \leq \frac{C}{r} \sup_{B_r(x)}|w_\ep(x) - w_0(x)| \leq \frac{C}{r}\mathfrak{r}(\ep)^2 \ \hbox{ for } \ B_r(x) \subset U \setminus (\Lambda_\ep \cup \Lambda_0).\]
We apply that estimate with $r = \mathfrak{r}(\ep)$ to obtain
\[ |\grad w_\ep(x) - \grad w_0(x)| \leq C \mathfrak{r}(\ep) \ \hbox{ for } \ d(x,\Lambda_\ep \cup \Lambda_0) \geq \mathfrak{r}(\ep).\] 
If $d(x, \Lambda_0 \cup \Lambda_\ep) \leq \mathfrak{r}(\ep)$ then the Hausdorff distance estimate established in the previous section implies
\[ d(x,\Lambda_0) \leq C\mathfrak{r}(\ep).\]
  The $C^{1,1}$ estimate of $w_0$ shows that
\[ w_0(x) \leq C\mathfrak{r}(\ep)^2, \ |\grad w_0(x)| \leq C\mathfrak{r}(\ep)\]
and the $L^\infty$ estimate of $w_\ep - w_0$ also shows
\[ w_\ep(x) \leq C\mathfrak{r}(\ep)^2.\]
We just need to establish an analogous supremum estimate of $\grad w_\ep(x)$.

First of all notice that if $w_\ep(x)>0$ then $\Delta w_\ep(x) = -\Delta \varphi_0(x)$ so $-w_\ep(x)\Delta w_\ep(x) <0$, while if $w_\ep(x) \leq 0$ then
\[ -w_\ep(x) \Delta w_\ep(x) \leq -w_\ep(x) (-\Delta \varphi_0(x)) \leq C\mathfrak{r}(\ep)^2\]
since $\Delta w_\ep \leq -\Delta \varphi_0$ everywhere.  Thus $-w_\ep(x) \Delta w_\ep(x) \leq Cr(\ep)^2$ in either case. 

Now we apply a Cacciopoli type estimate, take a standard cut-off function $\zeta$ which is $1$ in $B_{\mathfrak{r}(\ep)}(x)$ and zero outside of $B_{2\mathfrak{r}(\ep)}(x)$ with $|\grad \zeta| \leq \frac{C}{\mathfrak{r}(\ep)}$.  The argument of the previous paragraph shows that
\[ \dashint_{B_{2\mathfrak{r}(\ep)}(x)} -w_\ep \Delta w_\ep \zeta^2 dy \leq C\mathfrak{r}(\ep)^2\]
on the other hand we can compute
\[ \dashint_{B_{2\mathfrak{r}(\ep)}(x)} -w_\ep \Delta w_\ep \zeta^2 dy = \dashint_{B_{2\mathfrak{r}(\ep)}(x)}|\grad w_\ep|^2 \zeta^2 dy + \dashint_{B_{2\mathfrak{r}(\ep)}(x)}\zeta\grad w_\ep \cdot 2w_\ep \grad \zeta dy\]
by Young's inequality
\[ \dashint_{B_{2\mathfrak{r}(\ep)}(x)}\zeta\grad u_\ep \cdot 2w_\ep \grad \zeta dy \geq -\frac{1}{2}\dashint_{B_{2\mathfrak{r}(\ep)}(x)}|\grad w_\ep|^2 \zeta^2 dy- 2\dashint_{B_{2\mathfrak{r}(\ep)}(x)}w_\ep^2 |\grad \zeta|^2 dy\]
and
\[\dashint_{B_{2\mathfrak{r}(\ep)}(x)}w_\ep^2 |\grad \zeta|^2 dy \leq C\mathfrak{r}(\ep)^4\frac{1}{\mathfrak{r}(\ep)^2}. \]
Combining the previous inequalities leads to
\[\left(\dashint_{B_{\mathfrak{r}(\ep)}(x)}|\grad w_\ep|^2 dy\right)^{1/2} \leq C \mathfrak{r}(\ep).\]
Note that we have not proved a $\mathfrak{r}(\ep)$ rate for the gradient in $L^\infty$ even for the corrector problem, so this kind of averaged estimate is basically the best we can do without knowing something more about corrector problem.

\bibliographystyle{amsplain}
\bibliography{AbedinFeldman}

\providecommand{\bysame}{\leavevmode\hbox to3em{\hrulefill}\thinspace}
\providecommand{\MR}{\relax\ifhmode\unskip\space\fi MR }
\providecommand{\MRhref}[2]{%
  \href{http://www.ams.org/mathscinet-getitem?mr=#1}{#2}
}
\providecommand{\href}[2]{#2}
\begin{thebibliography}{10}

\bibitem{AleksanyanKuusi}
Gohar Aleksanyan and Tuomo Kuusi, \emph{Quantitative homogenization for the
  obstacle problem and its free boundary}, Preprint arXiv:2112.10879 (2022).

\bibitem{AttouchPicard2}
Hedi Attouch and Colette Picard, \emph{Asymptotic analysis of variational
  problems with constraints of obstacle type}, Publications Math\'{e}matiques
  d'Orsay 82 [Mathematical Publications of Orsay 82], vol.~7, Universit\'{e} de
  Paris-Sud, D\'{e}partement de Math\'{e}matiques, Orsay, 1982. \MR{683876}

\bibitem{AttouchPicard1}
H\'{e}dy Attouch and Colette Picard, \emph{Variational inequalities with
  varying obstacles: the general form of the limit problem}, J. Functional
  Analysis \textbf{50} (1983), no.~3, 329--386. \MR{695419}

\bibitem{Blank}
Ivan Blank, \emph{Sharp results for the regularity and stability of the free
  boundary in the obstacle problem}, Indiana Univ. Math. J. \textbf{50} (2001),
  no.~3, 1077--1112. \MR{1871348}

\bibitem{BlankLeCrone}
Ivan Blank and Jeremy LeCrone, \emph{Perturbed obstacle problems in {L}ipschitz
  domains: linear stability and nondegeneracy in measure}, Rocky Mountain J.
  Math. \textbf{49} (2019), no.~5, 1407--1418. \MR{4010568}

\bibitem{CaffarelliObstacleRevisited}
L.~A. Caffarelli, \emph{The obstacle problem revisited}, J. Fourier Anal. Appl.
  \textbf{4} (1998), no.~4-5, 383--402. \MR{1658612}

\bibitem{CaffarelliLee}
Luis Caffarelli and Ki-ahm Lee, \emph{Viscosity method for homogenization of
  highly oscillating obstacles}, Indiana Univ. Math. J. \textbf{57} (2008),
  no.~4, 1715--1741. \MR{2440878}

\bibitem{CaffarelliBUMI}
Luis~A. Caffarelli, \emph{A remark on the {H}ausdorff measure of a free
  boundary, and the convergence of coincidence sets}, Boll. Un. Mat. Ital. A
  (5) \textbf{18} (1981), no.~1, 109--113. \MR{607212}

\bibitem{CarboneColombini}
Luciano Carbone and Ferruccio Colombini, \emph{On convergence of functionals
  with unilateral constraints}, J. Math. Pures Appl. (9) \textbf{59} (1980),
  no.~4, 465--500. \MR{607050}

\bibitem{CioranescuMurat}
Doina Cioranescu and Fran\c{c}ois Murat, \emph{A strange term coming from
  nowhere [ {MR}0652509 (84e:35039a); {MR}0670272 (84e:35039b)]}, Topics in the
  mathematical modelling of composite materials, Progr. Nonlinear Differential
  Equations Appl., vol.~31, Birkh\"{a}user Boston, Boston, MA, 1997,
  pp.~45--93. \MR{1493040}

\bibitem{CodegoneRodrigues}
Marco Codegone and Jos\'{e}-Francisco Rodrigues, \emph{Convergence of the
  coincidence set in the homogenization of the obstacle problem}, Ann. Fac.
  Sci. Toulouse Math. (5) \textbf{3} (1981), no.~3-4, 275--285 (1982).
  \MR{658736}

\bibitem{DalMasoLongo}
Gianni Dal~Maso and Placido Longo, \emph{{$\Gamma $}-limits of obstacles}, Ann.
  Mat. Pura Appl. (4) \textbf{128} (1981), 1--50. \MR{640775}

\bibitem{DeGiorgi}
Ennio De~Giorgi, Gianni Dal~Maso, and Placido Longo, \emph{{$\Gamma $}-limits
  of obstacles}, Atti Accad. Naz. Lincei Rend. Cl. Sci. Fis. Mat. Nat. (8)
  \textbf{68} (1980), no.~6, 481--487. \MR{639976}

\bibitem{ElliotJanovsky}
C.~M. Elliott and V.~Janovsk\'{y}, \emph{A variational inequality approach to
  {H}ele-{S}haw flow with a moving boundary}, Proc. Roy. Soc. Edinburgh Sect. A
  \textbf{88} (1981), no.~1-2, 93--107. \MR{611303}

\bibitem{Kim}
Inwon~C. Kim, \emph{Error estimates on homogenization of free boundary
  velocities in periodic media}, Ann. Inst. H. Poincar\'{e} C Anal. Non
  Lin\'{e}aire \textbf{26} (2009), no.~3, 999--1019. \MR{2526413}

\bibitem{KimMellet}
Inwon~C. Kim and Antoine Mellet, \emph{Homogenization of a {H}ele-{S}haw
  problem in periodic and random media}, Arch. Ration. Mech. Anal. \textbf{194}
  (2009), no.~2, 507--530. \MR{2563637}

\bibitem{SerfatySerra}
Sylvia Serfaty and Joaquim Serra, \emph{Quantitative stability of the free
  boundary in the obstacle problem}, Anal. PDE \textbf{11} (2018), no.~7,
  1803--1839. \MR{3810473}

\end{thebibliography}

\end{document}